\documentclass[11pt]{article}
\usepackage{amsmath, amssymb, amsthm}
\usepackage{verbatim}
\usepackage{multicol}
\usepackage{enumerate}
\usepackage{comment}
\usepackage{dsfont}
\usepackage[none]{hyphenat}
\usepackage[unicode]{hyperref}
\hypersetup{
	colorlinks=true,
	linkcolor=blue,
	filecolor=magenta,      
	urlcolor=cyan,
	citecolor=blue
}
\usepackage{pgf}
\usepackage{tikz}
\usetikzlibrary{positioning,arrows,shapes,decorations.markings,decorations.pathreplacing,matrix,patterns}
\tikzstyle{vertex}=[circle,draw=black,fill=black,inner sep=0,minimum size=3pt,text=white,font=\footnotesize]

\usepackage{cleveref}

\date{}
\title{\vspace{-1.2cm} Large Cuts in Hypergraphs via Energy}

\author{Eero R\"aty\thanks{Ume\r{a} University, \emph{e-mail}: \textbf{eero.raty@umu.se}. Supported by a postdoctoral grant from the Osk.\ Huttunen Foundation.}, 
	Istv\'an Tomon\thanks{Ume\r{a} University, \emph{e-mail}: \textbf{istvan.tomon@umu.se}. Supported in part by the Swedish Research Council grant VR 2023-03375.}
}

\theoremstyle{plain}
\newtheorem{theorem}{Theorem}[section]
\newtheorem{corollary}[theorem]{Corollary}
\newtheorem{claim}[theorem]{Claim}
\newtheorem{lemma}[theorem]{Lemma}
\newtheorem{conjecture}[theorem]{Conjecture}

\Crefname{theorem}{Theorem}{Theorems}
\Crefname{definition}{Definition}{Definitions}
\Crefname{corollary}{Corollary}{Corollaries}
\Crefname{claim}{Claim}{Claims}
\Crefname{lemma}{Lemma}{Lemmas}
\Crefname{conjecture}{Conjecture}{Conjectures}
\Crefname{problem}{Problem}{Problems}
\Crefname{prop}{Proposition}{Propositions}

\theoremstyle{definition}
\newtheorem{definition}{Definition}

\DeclareMathOperator{\tr}{tr}
\DeclareMathOperator{\surp}{sp}
\DeclareMathOperator{\mc}{mc}

\begin{document}
	
	\maketitle
	\sloppy

\begin{abstract}
A simple probabilistic argument shows that every $r$-uniform hypergraph with $m$ edges contains an $r$-partite subhypergraph with at least $\frac{r!}{r^r}m$ edges. The celebrated result of Edwards states that in the case of graphs, that is $r=2$, the resulting bound $m/2$ can be improved to $m/2+\Omega(m^{1/2})$, and this is sharp. We prove that if $r\geq 3$, then there is an $r$-partite subhypergraph with at least $\frac{r!}{r^r} m+m^{3/5-o(1)}$ edges.  Moreover, if the hypergraph is linear, this can be improved to $\frac{r!}{r^r} m+m^{3/4-o(1)},$ which is tight up to the $o(1)$ term. These improve results of Conlon, Fox, Kwan, and Sudakov. Our proof is based on a combination of probabilistic, combinatorial, and linear algebraic techniques, and semidefinite programming.

A key part of our argument is relating the \emph{energy} $\mathcal{E}(G)$ of a graph $G$ (i.e.\ the sum of absolute values of eigenvalues of the adjacency matrix) to its  maximum cut. We prove that every $m$ edge multigraph $G$ has a cut of size at least $m/2+\Omega(\frac{\mathcal{E}(G)}{\log m})$, which might be of independent interest.
\end{abstract}
 
\section{Introduction}

Given a graph $G$, a \emph{cut} in $G$ is a partition of the vertices into two parts, together with all the edges having exactly one vertex in each of the parts. The size of the cut is the number of its edges. The MaxCut problem asks for the maximum size of a cut, denoted by $\mc(G)$, and it is a central problem both in discrete mathematics and theoretical computer science \cite{AlonMaxCut,AKS05,BJS,Edwards1,Edwards2,GJS,GW95}. If $G$ has $m$ edges, a simple probabilistic argument shows that the maximum cut is always at least $m/2$, so it is natural to study the \emph{surplus}, defined as $\surp(G)=\mc(G)-m/2$. A fundamental result in the area is due to Edwards \cite{Edwards1,Edwards2} stating that  every graph $G$ with $m$ edges satisfies $\surp(G)\geq \frac{\sqrt{8m+1}-1}{8}=\Omega(m^{1/2})$ and this is tight when $G$ is the complete graph on an odd number of vertices. The study of MaxCut in graphs avoiding a fixed graph $H$ as a subgraph was initiated by Erd\H{o}s and Lov\'asz (see \cite{erdos}) in the 70's, and substantial amount of research was devoted to this problem since then, see e.g.\ \cite{AlonMaxCut,AKS05,BJS,GJS,GW95}.
 
 The MaxCut problem can be naturally extended to hypergraphs. If $H$ is an $r$-uniform hypergraph (or $r$-graph, for short), a \emph{$k$-cut} in $H$ is a partition of the vertex set into $k$ parts, together with all the edges having at least one vertex in each part.  Similarly as before, $\mc_k(H)$ denotes the maximum number of edges in a $k$-cut. This notion was first considered by Erd\H{o}s and Kleitman \cite{EK68} in 1968, and they observed that if $H$ has $m$ edges, then  a random $k$-cut (i.e., where every vertex is assigned to any of the parts independently with probability $1/k$) has $\frac{S(r,k)k!}{k^r}m$ edges in expectation, where $S(r,k)$ is the Stirling number of the second kind, i.e.\ the number of unlabelled partitions of $\{1,\dots,r\}$ into $k$ parts. In computer science, the problem of computing the 2-cut is known as \emph{max set splitting}, or $Er$-set splitting if the hypergraph is $r$-uniform, see e.g.\ \cite{guruswami,hastad}.
 
 Similarly as before, we define the \emph{$k$-surplus} of an $r$-graph $H$ as $$\surp_k(H):=\mc_k(H)-\frac{S(r,k)k!}{k^r}m,$$ that is, the size of the maximum $k$-cut of $H$ above the expectation of the random cut. A simple extension of the methods used for graphs shows that  Edwards' lower bound holds for hypergraphs as well, that is, $\surp_k(H)=\Omega_r(m^{1/2})$. However, as proved by Conlon, Fox, Kwan, and Sudakov \cite{CFKS},  this can be significantly improved, unless $(r,k)=(3,2)$.

Consider the case $r=3$ and $k=2$. Let $H$ be a 3-graph and let $G$ be the underlying multi-graph of $H$, that is, we connect every pair of vertices in $G$ by an edge as many times as it appears in an edge of $H$. It is easy to see that if $X\cup Y$ is a partition of $V(H)$, then $H$ has half the number of edges in the 2-cut $(X,Y)$ as $G$. Hence, the problem of 2-cuts in 3-graphs reduces to a problem about  cuts in multi-graphs, and thus not much of an interest from the perspective of hypergraphs. In particular, in case $H$ is a Steiner triple system, then $G$ is the complete graph, thus $\surp_2(H)=\Theta(\sqrt{m})$. Therefore, the bound of Edwards is sharp in this case. However, as it was proved by Conlon, Fox, Kwan, and Sudakov \cite{CFKS}, in the case of 3-cuts, we are guaranteed much larger surplus, in particular $\surp_3(H)=\Omega_r(m^{5/9})$. We further improve this in the following theorem.
 
\begin{theorem}\label{thm:main_cut}
    Let $H$ be a $3$-graph with $m$ edges. Then $H$ has a 3-cut of size at least
    $$\frac{2}{9}m+\Omega\left(\frac{m^{3/5}}{(\log m)^2}\right).$$
\end{theorem}
 
 If $H$ is the complete 3-graph on $n$ vertices, then $m=\binom{n}{3}$ and $\surp_3(H)=O(n^2)$, so $\surp_3(H)$ might be as small as $O(m^{2/3})$. It is conjectured by Conlon, Fox, Kwan, and Sudakov \cite{CFKS} that this upper bound is sharp. A hypergraph is \emph{linear} if any two of its edges intersect in at most one vertex. We prove the following stronger bound for linear 3-graphs.

\begin{theorem}\label{thm:main_cut_linear}
    Let  $H$ be a linear $3$-graph with $m$ edges.  Then $H$ has a 3-cut of size at least
    $$\frac{2}{9}m+\Omega\left(\frac{m^{3/4}}{(\log m)^2}\right).$$
\end{theorem}

 In case the edges of the $n$ vertex $3$-graph $H$ are included independently with probability $n^{-1}$, then $H$ is close to being linear in the sense that no pair of vertices is contained in more than $O(\log n)$ edges with positive probability.  Also, it is not hard to show that $\surp(H)=O(m^{3/4})$ with high probability, see Section \ref{sect:prob} for a detailed argument. We believe that more sophisticated random 3-graph models can produce linear hypergraphs with similar parameters, however, we do not pursue this direction. 
 
Theorems \ref{thm:main_cut} and \ref{thm:main_cut_linear} can be deduced, after a bit of work, from the following general bound on the surplus of 3-graphs, which takes into account the maximum degree, and maximum co-degree.

\begin{theorem}\label{thm:main}
 Let $H$ be a 3-uniform  multi-hypergraph. If $H$ contains an induced subhypergraph with $m$ edges, maximum degree $\Delta$ and maximum co-degree $D$, then
    $$\surp_3(H)\geq \Omega\left(\frac{m}{\sqrt{\Delta D}(\log m)^2}\right).$$
\end{theorem}

 Consider the random 3-graph on $n$ vertices, in which each edge is chosen independently with probability $p$, where $\log n/n^2\ll p\ll 1/n$. With positive probability, this hypergraph satisfies $\Delta=\Theta(pn^2)$, $D=O(\log n)$, $m=\Theta(pn^3)$ and $\surp_3(H)=O(\sqrt{p}n^2)$, see Section \ref{sect:prob} for a detailed argument. The upper bound on the surplus of this 3-graph coincides with the lower bound of Theorem \ref{thm:main} up to logarithmic factors, showing that our theorem is tight for a large family of parameters.

\bigskip
 
 Now let us consider $k$-cuts in $r$-graphs for $r\geq 4$ and $2\leq k\leq r$. In this case, Conlon, Fox, Kwan, and Sudakov \cite{CFKS} established the same lower bound  $\surp_k(H)=\Omega_r(m^{5/9})$. From above, they showed that if $H$ is the random $r$-graph on $n$ vertices, where each edge is included with probability $p=\frac{1}{2}n^{-(r-3)}$, then $\surp_k(H)=O_r(m^{2/3})$ with high probability. This disproved the conjecture of Scott \cite{Scott} that complete $r$-graphs minimize the maximum 2-cut. Moreover, it is shown in \cite{CFKS} that any general lower bound one gets for the 3-cut problem for 3-graphs extends for the $k$-cut problem for $r$-graphs if  $k=r-1$ or $r$. In particular, we achieve the following improvements in these cases.

 \begin{theorem}\label{thm:rk}
 Let $r\geq 4$ and $k\in \{r-1,r\}$. Let $H$ be an $r$-graph with $m$ edges. Then
    $$\surp_k(H)=\Omega_r\left(\frac{m^{3/5}}{(\log m)^2}\right).$$
 Moreover, if $H$ is linear, then
     $$\surp_k(H)=\Omega_r\left(\frac{m^{3/4}}{(\log m)^2}\right).$$
 \end{theorem}

 \bigskip

\noindent
\textbf{Proof overview.} Let us give a brief outline of the proof of Theorem \ref{thm:main}, from which all other theorems are deduced. Our approach is fairly different from that of Conlon, Fox, Kwan, and Sudakov \cite{CFKS}, which is based on probabilistic and combinatorial ideas. We use a combination of probabilistic and spectral techniques to bound the surplus.

First, we show that the surplus in a (multi-)graph can be lower bounded by the \emph{energy of the graph}, i.e.\ the sum of absolute values of eigenvalues of the adjacency matrix. This quantity is extensively studied in spectral graph theory, originally introduced in theoretical chemistry. We follow the ideas of R\"aty, Sudakov, and Tomon \cite{RST} to estimate the surplus by a semidefinite program. Then, we construct  a positive semidefinite matrix based on the spectral decomposition of the adjacency matrix of the graph to  show that the energy is a lower bound for the corresponding program. This can be found in Section \ref{sect:energy}. See also the recent result of Sudakov and Tomon \cite{logrank} on the Log-rank conjecture for a similar idea executed.

Now let $H$ be a 3-uniform hypergraph for which we wish to find a large 3-cut. First, we sample a third of the vertices randomly, and write $\mathbf{X}$ for the set of sampled vertices. For each $e\in E(H)$ that contains exactly one element of $\mathbf{X}$, we remove this one element, and set $G^*$ to be the multi-graph of such edges. We argue that the expected energy of $G^*$ is large. In order to do this, we view $H$ as a colored multi-graph $G$, in which each pair of vertices $\{u,v\}$ is included  as many times as it appears in an edge $\{u,v,w\}$ of $H$, and whose color is the third vertex $w$. Let $G_1^*$ be the subgraph of $G$ in which we keep those edges whose color is in $\mathbf{X}$. Then $G^*$ is an induced subgraph of $G_1^*$, so by interlacing, it inherits certain spectral properties of $G^*_1$. Using matrix concentration inequalities, we show that if we randomly sample the colors in a colored multi-graph $G$, the resulting graph $G_1^*$ is a good spectral approximation of the original graph, assuming each color class is a low-degree graph. This, surprisingly, ensures that the energy of $G^*_1$ is large with high probability, even if $G$ had small energy. An explanation for this and the detailed argument can be found in Section \ref{sect:colored}.

Now if $G^*$ has  large energy, it has a large cut $\mathbf{Y}\cup \mathbf{Z}$ as well. We conclude that $\mathbf{X}\cup\mathbf{Y}\cup \mathbf{Z}$ is a partition of the vertex set of $H$ with the required number of edges. We unite all parts of our argument and prove Theorems \ref{thm:main_cut}, \ref{thm:main_cut_linear}, \ref{thm:main}, and \ref{thm:rk} in Section \ref{sect:proofs}.

 \section{Preliminaries}
In this section, we introduce the notation used throughout our paper, which is mostly conventional, and present some basic results. 

Given a vector $x\in \mathbb{R}^n$, we write $||x||=||x||_2$ for the Euclidean norm. If $A\in \mathbb{R}^{n\times n}$ is a symmetric matrix, let $\lambda_i(A)$ denote the $i$-th largest eigenvalue of $A$. The \emph{spectral radius} of $A$ is denoted by $||A||$, and it can be defined in a number of equivalent ways: 
$$||A||=\max\{|\lambda_1(A)|,|\lambda_n(A)|\}=\max_{x\in \mathbb{R}^n, ||x||=1} ||Ax||.$$
The \emph{Frobenius norm} of $A$ is defined as
 $$||A||_F^2:=\langle A,A\rangle=\tr(A^2)=\sum_{i=1}^{n}\lambda_i(A)^2,$$
 where $\langle A,B\rangle=\sum_{i,j}A(i,j)B(i,j)$ is the entrywise scalar product on the space of matrices.
 
We make use of two classical results about the eigenvalues of symmetric matrices, Weyl's inequality and the Cauchy interlacing theorem.

\begin{lemma}[Weyl's inequality \cite{Weyl}]\label{Weyl}
 Let $A,B\in \mathbb{R}^{n\times n}$  be symmetric matrices. If $i+j-1\leq n$, then 
 $$\lambda_{i+j-1}(A+B)\leq \lambda_i(A)+\lambda_j(B),$$
 and if $i+j>n$, then
 $$\lambda_i(A)+\lambda_j(B)\leq \lambda_{i+j-n}(A+B).$$
 In particular, 
 $$|\lambda_k(A+B)-\lambda_k(A)|\leq ||B||.$$
 \end{lemma}

 \begin{lemma}[Cauchy interlacing theorem]
 Let $A\in \mathbb{R}^{n\times n}$ be a symmetric matrix, and let $B$ be an $N\times N$ sized principal submatrix of $A$. Then 
 $$\lambda_i(A)\geq \lambda_i(B)\geq \lambda_{i+(n-N)}(A).$$
 In particular,
 $$||B||\leq ||A||.$$
 \end{lemma}






\subsection{Basics of cuts}

In this section, we present some basic results and definitions about cuts.

It will be convenient to work with $r$-uniform multi-hypergraphs $H$ (or $r$-multigraph for short). That is, we allow multiple edges on the same $r$ vertices. A \emph{$k$-cut} in $H$ is partition of $V(H)$ into $k$ parts $V_1,\dots,V_k$, together with all the edges that have at least one vertex in each of the parts. We denote by $e(V_1,\dots,V_k)$ the number of such edges.

\begin{definition}
    The Max-$k$-Cut of $H$ is the maximum size of a $k$-cut, and it is denoted by $\mc_k(H)$. 
\end{definition}

In case we consider a random partition, i.e.\ when each vertex of $H$ is assigned to one of $V_i$ independently with probability $1/k$, the expected size of a cut is $\frac{S(r,k)k!}{k^r}e(H)$, where $S(r,k)$ is the number of unlabelled partitions of $\{1,\dots,r\}$ into $k$ parts.

\begin{definition}
    The \emph{$k$-surplus} of $H$ is $$\surp_k(H)=\mc_k(H)-\frac{S(r,k)k!}{k^r}e(H).$$
\end{definition} 

\noindent
Clearly,  $\surp_k(H)$ is always nonnegative. Next, we recall several basic facts from \cite{CFKS}.

\begin{lemma}[Corollary 3.2 in \cite{CFKS}]\label{lemma:m/n}
Let $H$ be an $r$-multigraph on $n$ vertices. Then for every $2\leq k\leq r$, 
$$\surp_k(H)=\Omega_r\left(\frac{e(H)}{n}\right).$$
\end{lemma}

\begin{lemma}[Theorem 1.4 in \cite{CFKS}]\label{lemma:surplus>n}
Let $H$ be an $r$-multigraph on $n$ vertices with no isolated vertices. Then for every $2\leq k\leq r$,
$$\surp_{k}(H)=\Omega_r(n).$$
\end{lemma}

\noindent
In the case of multi-graphs $G$, we write simply $\mc(G)$ and $\surp(G)$ instead of $\mc_2(G)$ and $\surp_2(G)$.

\begin{lemma}\label{lemma:surplus_sum}
Let $G$ be a multi-graph and let $V_1,\dots,V_k\subset V(G)$ be disjoint sets. Then
$$\surp(G)\geq \sum_{i=1}^k \surp(G[V_i]).$$
\end{lemma}

Given an $r$-multigraph $H$, for every $q\leq r$, the \emph{underlying} $q$-multigraph is the $q$-multigraph on vertex set $V(H)$, in which each $q$-tuple is added as an edge as many times as it is contained in an edge of $H$. The final lemma in this section is used to deduce Theorem \ref{thm:rk} from our results about 3-graphs.

\begin{lemma}\label{lemma:underlying}
Let $H$ be an $r$-multigraph, and let $H'$ be the underlying $(r-1)$-multigraph. Then
$$2r\surp_{r}(H)\geq \surp_{r-1}(H)\geq 2\surp_{r-1}(H').$$
\end{lemma}

\begin{proof}
We start with the first inequality. Let $U_1,\dots,U_{r-1}$ be a partition of $V(H)$ such that $e_{H}(U_1,\dots,U_{r-1})=\mc_{r-1}(H)$. Let $V_{r}$ be a random subset of $V(H)$, each element chosen independently with probability $1/r$, and let $V_i=U_i\setminus V_r$ for $i=1,\dots,r-1$. Note that if an edge $e$ of $H$ contains at least one element of each $U_1,\dots,U_{r-1}$, then exactly one of these sets contains two vertices of $e$. Hence, the probability that $e$ is cut by $V_1,\dots,V_{r}$ is $c_r=2(1-\frac{1}{r})^{r-1}\frac{1}{r}\geq \frac{1}{2r}$. Hence,
$$\mathbb{E}(e_H(V_1,\dots,V_r))= c_r\mc_{r-1}(H)=c_r\surp_{r-1}(H)+c_r\frac{r!/2}{(r-1)^{r-1}}e(H)=c_r\surp_{r-1}(H)+\frac{r!}{r^r}e(H),$$
so $\surp_{r}(H)\geq c_r \surp_{r-1}(H)$.

Now let us prove the second inequality. Let $V_1,\dots,V_{r-1}$ be a partition of $V(H)$ such that $e_{H'}(V_1,\dots,V_{r-1})=\mc_{r-1}(H')$. If $e$ is an edge of $H$ cut by $V_1,\dots,V_{r-1}$, then among the $(r-1)$-element subsets of $e$, exactly two are cut by $V_1,\dots,V_{r-1}$. Thus, $e_H(V_1,\dots,V_{r-1})=2e_{H'}(V_1,\dots,V_{r-1})$. Therefore, 
$$\mc_{r-1}(H)\geq 2\mc_{r-1}(H')=2\surp_{r-1}(H')+\frac{(r-1)!}{(r-1)^{r-1}}e(H')=2\surp_{r-1}(H')+\frac{r!/2}{(r-1)!}e(H).$$
Hence, $\surp_{r-1}(H)\geq 2\surp_{r-1}(H').$
\end{proof}

\subsection{Probabilistic constructions}\label{sect:prob}

In this section, we verify our claims about the surplus of random 3-graphs.

\begin{lemma}\label{lemma:construction}
For every $\frac{10\log n}{n^2}\leq p\leq \frac{1}{n}$, there exists a 3-graph $H$ on $n$ vertices such that $H$ has $\Theta(pn^3)$ edges, the maximum degree of $H$ is $O(pn^2)$, the maximum co-degree of $H$ is $O(\log n)$, and $\surp_3(H)=O(\sqrt{p}n^2)$.
\end{lemma}

\begin{proof}
Let $p$ be such that $\frac{10\log n}{n^2}\leq p\leq \frac{1}{n}$, and let $H$ be the 3-graph on $n$ vertices, in which each of the $\binom{n}{3}$ potential edges are included independently with probability $p$. We assume that $n$ is sufficiently large. By the Chernoff-Hoeffding theorem, for every $x\in (0,1/2)$, we have
$$\mathbb{P}\left(|e(H)/\binom{n}{3}-p|\geq x\right)\leq 2\exp\left(-\frac{x^2\binom{n}{3}}{4p}\right).$$
Hence, setting $x=20\sqrt{p}n^{-3/2}$,  $$\mathbb{P}\left(|e(H)-p\binom{n}{3}|\geq 20\sqrt{p}n^{3/2}\right)\leq 0.1.$$ 
 Also, by similar concentration arguments, the maximum degree of $H$ is at most $2p\binom{n}{2}$ with probability $0.9$ by our assumption $pn^2\geq 10\log n$. Moreover, the maximum co-degree is at most $10\log n$ with probability $0.9$ by our assumption that $p\leq 1/n$. 

Let $X,Y,Z$ be a partition of the vertex set into three parts, and let $N=(n/3)^3\geq |X||Y||Z|=T$. For every $x>0$, 
$$\mathbb{P}\left(e(X,Y,Z)\geq N(p+x)\right)\leq \mathbb{P}\left(\frac{e(X,Y,Z)}{T}\geq p+x\frac{N}{T}\right)\leq \exp\left(-\frac{(xN/T)^2\cdot T}{4p}\right)\leq  \exp\left(-\frac{x^2 N}{4p}\right),$$
where the second inequality is due to the Chernoff-Hoeffding theorem. Set $x=\frac{20\sqrt{p}}{n}$. The number of partitions of $V(H)$ into three sets is at most $3^n$, so
$$\mathbb{P}(\exists X,Y,Z: e(X,Y,Z)\geq N(p+x))\leq 3^n\exp\left(-\frac{x^2 N}{4p}\right)<0.1.$$
Hence, with probability 0.9, $\mc_3(H)\leq N(p+x)\leq pN+3\sqrt{p}n^2$.

In conclusion, with positive probability, there exists a 3-graph $H$ on $n$ vertices such that $|e(H)-p\binom{n}{3}|\leq O(\sqrt{p}n^{3/2})$, the maximum degree of $H$ is at most $pn^2$, the maximum co-degree is at most $10\log n$, and $\mc_3(H)\leq p\frac{n^3}{27}+O(\sqrt{p}n^2)$. For such a 3-graph $H$, we have
$$\surp_3(H)=\mc_3(H)-\frac{2}{9}e(H)\leq p\frac{n^3}{27}+O(\sqrt{p}n^2)-\frac{2}{9}p\binom{n}{3}+O(\sqrt{p}n^{3/2})=O(\sqrt{p}n^2).$$
In the last equality, we used that $p\cdot \frac{n^3}{27}-\frac{2}{9}p\binom{n}{3}=O(pn^2)$.
\end{proof}

If one chooses $p=1/n$, we have $m=e(H)=\Theta(n^2)$, every co-degree in $H$ is $O(\log n)$, and $\surp_3(H)=O(n^{3/2})=O(m^{3/4})$, showing the almost tightness of Theorem \ref{thm:main_cut_linear}.

\section{Surplus and Energy}\label{sect:energy}

 In this section, we prove a lower bound on the surplus in terms of the energy of the graph. Let $G$ be a multi-graph on vertex set $V$, and let $|V|=n$. The adjacency matrix of $G$ is the matrix $A\in \mathbb{R}^{V\times V}$ defined as $A(u,v)=k$, where $k$ is the number of edges between $u$ and $v$.
 
 Recall that $\mc(G)=\frac{e(G)}{2}+\surp(G)$. Next, we claim that 
 \begin{equation*}
      \surp(G)=\max_{x\in \{-1,1\}^{V}} -\frac{1}{2}\sum_{uv\in E(G)}x(u)x(v).
 \end{equation*}
 Indeed, every $x\in \{-1,1\}^{V}$ corresponds to a partition $X,Y$ of $V(G)$, where $X=\{v\in V(G):x(v)=1\}$, and then
 $$e(X,Y)-\frac{m}{2}=\frac{1}{2}(e(X,Y)-e(X)-e(Y))=-\frac{1}{2}\sum_{uv\in E(G)}x(u)x(v).$$
 We define the surplus of arbitrary symmetric matrices $A\in\mathbb{R}^{n\times n}$ as well. We restrict our attention to symmetric matrices, whose every diagonal entry is 0. Let
\begin{equation}\label{equ:quadratic2}
      \surp(A)=\max_{x\in \{-1,1\}^{n}} -\frac{1}{2}\sum_{i,j\in [n]}A(i,j)x(i)x(j).
 \end{equation}
 Then $\surp(G)=\surp(A)$ if $A$ is the adjacency matrix of $G$. Observe that by the condition that every diagonal entry of $A$ is zero, $\sum_{i,j}A(i,j)x(i)x(j)$ is a multilinear function, so its minimum on $[-1,1]^n$ is attained by one of the extremal points. Therefore, (\ref{equ:quadratic2}) remains true if the maximum is taken over all $x\in [-1,1]^{n}$. We introduce the semidefinite relaxation of $\surp(A)$. Let 
\begin{equation}\label{equ:semidefinite}
    \surp^*(A)=\max-\frac{1}{2}\sum_{i,j}A(i,j)\langle z_i,z_j\rangle,
\end{equation}
where the maximum is taken over all $z_1,\dots,z_n\in\mathbb{R}^{n}$ such that $||z_i||\leq 1$ for $i\in [n]$. Clearly, $\surp^*(A)\geq \surp(A).$ On the other hand, we can use the following symmetric  analogue of Grothendieck's inequality \cite{Groth} to show that $\surp^{*}(A)$ also cannot be much larger than $\surp(A)$.
	
	\begin{lemma}[\cite{AMMN,CW04}]\label{lemma:groth}
		There exists a universal constant $C>0$ such that the following holds. Let $M\in \mathbb{R}^{n\times n}$ be symmetric. Let 
		$$\beta=\sup_{x\in [-1,1]^{n}}\sum M_{i,j}x(i)x(j),$$
		and let
		$$\beta^{*}=\sup_{v_1,\dots,v_n\in \mathbb{R}^{n},||v_i||\leq 1}\sum M_{i,j}\langle v_i,v_j\rangle.$$
		Then $\beta\leq \beta^{*}\leq C\beta \log n$.
	\end{lemma}

 Applying this lemma with the matrix $M=-\frac{1}{2}A$, we get that 
 $$\surp^*(A)\leq C\surp(A)\log n.$$
 Finally, it is convenient to rewrite equation (\ref{equ:semidefinite}) as
 \begin{equation}\label{equ:matrix}
      \surp^*(A)=\max-\frac{1}{2}\langle A,X\rangle,
 \end{equation}
 where the maximum is taken over all positive semidefinite symmetric matrices $X\in \mathbb{R}^{n\times n}$ which satisfy $X(i,i)\leq 1$ for every $i\in [n]$. Indeed, (\ref{equ:semidefinite}) and (\ref{equ:matrix}) are equivalent, as if $z_1,\dots,z_n\in\mathbb{R}^n$ such that $||z_i||\leq 1$, then the matrix $X$ defined as $X(i,j)=\langle z_i,z_j\rangle$ is positive semidefinite and satisfies $X(i,i)=||z_i||^2\leq 1$. Also, every positive semidefinite matrix $X$ such that $X(i,i)\leq 1$ is the Gram matrix of some vectors $z_1,\dots,z_n\in \mathbb{R}^n$ satisfying $||z_i||\leq 1$.

 In the next lemma, we connect the spectral properties of $A$ to the surplus. Let $\lambda_1\geq \dots\geq \lambda_n$ be the eigenvalues of $A$. 
 
 \begin{definition}
     The \emph{energy} of $A\in \mathbb{R}^{n\times n}$ is defined as 
$$\mathcal{E}(A)=\sum_{i=1}^{n}|\lambda_i|,$$
and if $A$ is the adjacency matrix of a graph $G$, then the energy of $G$ is $\mathcal{E}(G)=\mathcal{E}(A)$. 
 \end{definition}

Usually, the energy of $G$ is denoted by $E(G)$, but we use the notation $\mathcal{E}(G)$ to not confuse with the edge set of $G$, or with the expectation $\mathbb{E}$. In a certain weak sense, the energy measures how random-like a graph is. Indeed, if $G$ is a $d$-regular graph on $n$ vertices, then the energy of $G$ is at most $O(\sqrt{d} n)$, and this bound is attained by the random $d$-regular graph. Next, we show that the energy is a lower bound (up to a constant factor) for $\surp^*(G)$.

\begin{lemma}
If $\tr(A)=0$, then $\surp^*(A)\geq \frac{1}{4}\mathcal{E}(A).$
\end{lemma}

\begin{proof}
Let $v_1,\dots,v_n$ be an orthonormal set of eigenvectors of $A$ such that $\lambda_i$ is the eigenvalue corresponding to $v_i$. Let 
$$X=\sum_{i:\lambda_i<0}v_i\cdot v_i^T.$$
Then $X$ is positive semidefinite. Also, for $j=1,\dots,n$,
$$X(j,j)=\sum_{i:\lambda_i<0}v_i(j)^2\leq \sum_{i=1}^n v_i(j)^2=1.$$ Hence, by (\ref{equ:matrix}), 
$$\surp^{*}(A)\geq -\frac{1}{2}\langle X,A\rangle.$$
Observe that $A=\sum_{i=1}^{n}\lambda_i (v_i\cdot v_i^T)$. Hence,
$$-\frac{1}{2}\langle X,A\rangle=-\frac{1}{2}\left\langle \sum_{i:\lambda_i<0}v_i v_i^T,\sum_{i=1}^n\lambda_iv_i v_i^T\right\rangle=-\frac{1}{2}\sum_{i:\lambda_i<0}\sum_{j=1}^{n}\lambda_j\langle v_i,v_j\rangle^2=-\frac{1}{2}\sum_{i:\lambda_i<0}\lambda_i.$$
 As $\sum_{i=1}^{n}\lambda_i=\mbox{tr}(A)=0$, we have $\sum_{i:\lambda_i<0}\lambda_i=-\frac{1}{2}\mathcal{E}(A)$. This finishes the proof.
\end{proof}

\begin{corollary}\label{cor:surplus_energy}
If $A$ has only zeros in the diagonal, then $\surp(A)=\Omega(\frac{\mathcal{E}(A)}{\log n}).$
\end{corollary}

We finish this section by stating a simple result about the additivity properties of energy, which will come in handy later.

\begin{lemma}\label{lemma:energy_additive}
Let $A,B\in \mathbb{R}^{n\times n}$ be symmetric matrices. Then
$$\mathcal{E}(A+B)\leq 4(\mathcal{E}(A)+\mathcal{E}(B)).$$
\end{lemma}

\begin{proof}
Let $m$ be the largest index such that $\lambda_m(A+B)>0$. By Weyl's inequality, for $i=1,\dots,n$, we have
$$\lambda_i(A+B)\leq \lambda_{\lceil i/2\rceil}(A)+\lambda_{i+1-\lceil i/2\rceil}(B)\leq \lambda_{\lceil i/2\rceil}(A)+\lambda_{\lceil i/2\rceil}(B),$$
hence
$$\sum_{i=1}^{m}|\lambda_{i}(A+B)|\leq \sum_{i=1}^{m}|\lambda_{\lceil i/2\rceil}(A)|+|\lambda_{\lceil i/2\rceil}(B)|\leq \sum_{i=1}^{\lceil m/2\rceil}2|\lambda_i(A)|+2|\lambda_i(B)|\leq 2(\mathcal{E}(A)+\mathcal{E}(B)).$$
We can bound the sum of negative eigenvalues of $A+B$ similarly, so in total
$$\mathcal{E}(A+B)=\sum_{i=1}^{n}|\lambda_{i}(A+B)|\leq 4(\mathcal{E}(A)+\mathcal{E}(B)).$$
\end{proof}

\section{Sampling colored multi-graphs}\label{sect:colored}

This section is concerned with edge-colored multi-graphs. Our goal is to show that if one randomly samples the colors in such a graph, then the resulting graph has large cuts. We use bold letters to distinguish random variables from deterministic ones. In the rest of this section, we work with the following setup.

\medskip

\noindent
\textbf{Setup.} Let $G$ be a multi-graph with $n$ vertices and $m$ edges, and let $\phi:E(G)\rightarrow \mathbb{N}$ be a coloring of its edges. Let $A$ be the adjacency matrix of $G$. Let $\Delta$ be the maximum degree of $G$, and assume that every color appears at most $D$ times at every vertex for some $D\leq \Delta$. Let $p\in (\frac{100}{m},0.9)$, and sample the colors appearing in $G$ independently with probability $p$. Let $\mathbf{H}$ be the sub-multi-graph of $G$ of the sampled colors, and let $\mathbf{B}$ be the adjacency matrix of $\mathbf{H}$. 

\medskip

 In our first lemma, we show that $pA$ is a good spectral approximation of $\mathbf{B}$. In order to do this, we use a result of Oliviera \cite{oliviera} on the concentration of matrix martingales. More precisely, we use the following direct consequence of Theorem 1.2 in \cite{oliviera}.

\begin{lemma}[Theorem 1.2 in \cite{oliviera}]\label{lemma:matrix_concentration}
Let $\mathbf{M}_1,\dots,\mathbf{M}_m$ be independent random $n\times n$ symmetric matrices, and let $\mathbf{P}=\mathbf{M}_1+\dots+\mathbf{M}_m$. Assume that $||\mathbf{M}_i||\leq D$ for $i=1,\dots,m$, and define
$$W=\sum_{i=1}^m\mathbb{E}(\mathbf{M}_i^2).$$
Then for all $t>0$,
$$\mathbb{P}(||\mathbf{P}-\mathbb{E}(\mathbf{P})||\geq t)\leq 2n\exp\left(-\frac{t^2}{8||W||+4Dt}\right).$$
\end{lemma}

\begin{lemma}\label{lemma:colored_sampling}
For every $t\geq 0$, 
$$\mathbb{P}\left(||pA-\mathbf{B}||\geq t\right)\leq 2n\exp\left(-\frac{t^2}{8p\Delta D+4Dt}\right).$$
\end{lemma}

\begin{proof}
Let $1,\dots,m$ be the colors of $G$, let $A_i$ be  the adjacency matrix of color $i$, and let $\mathbf{I}_i\in \{0,1\}$ be the indicator random variable that color $i$ is sampled. Then setting $\mathbf{M}_i=\mathbf{I}_i A_i$, we have $\mathbf{B}=\mathbf{M}_1+\dots+\mathbf{M}_n$. Moreover, $||\mathbf{M}_i||\leq ||A_i||\leq D$, as the spectral radius is upper bounded by the maximum $\ell_1$-norm of the row vectors, which in turn is the maximum degree of the graph of color $i$. Now let us bound $||W||$, where $$W=\sum_{i=1}^{m}\mathbb{E}(\mathbf{M}_i^2)=p\sum_{i=1}^m A_i^2.$$
Let $\mu_1\geq \dots\geq \mu_n\geq 0$ be the eigenvalues of $W$. Then for every $k$,
$$\sum_{i=1}^n \mu_i^k=\tr(W^k)=\sum_{v_1,\dots,v_k\in V(G)}W(v_1,v_2)\dots W(v_{k-1},v_k)W(v_k,v_1).$$
Observe that $A_i^2(a,b)$ is the number of walks of length 2 between $a$ and $b$ using edges of color $i$, so $W(a,b)/p=\sum_{i=1}^m A^2(a,b)$ is the number of monochromatic walks of length 2 between $a$ and $b$ in $G$. Hence, $\tr(W^k/p^k)$ is the number of closed walks $(v_i,e_i,w_i,f_i)_{i=1,\dots,k}$, where $v_i,w_i$ are vertices, $e_i,f_i$ are edges, $v_i,w_i$ are endpoints of $e_i$, $w_i,v_{i+1}$ are endpoints of $f_i$ (where the indices are meant modulo $k$), and $e_i$ and $f_i$ has the same color. The total number of such closed walks is upper bounded by $n\Delta^k D^k$, as there are $n$ choices for $v_1$, if $v_i$ is already known there are at most $\Delta$ choices for $e_i$ and $w_i$, and as the color of $f_i$ is the same as of $e_i$, there are at most $D$ choices for $f_i$ and $v_{i+1}$. Hence, we get
$$\mu_1^k/p^k\leq n\Delta^k D^k.$$
As this holds for every $k$, we deduce that $||W||=\mu_1\leq p\Delta D$. But then the inequality 
$$\mathbb{P}\left(||pA-\mathbf{B}||\geq t\right)\leq 2n\exp\left(-\frac{t^2}{8p\Delta D+4Dt}\right)$$
is a straightforward consequence of Lemma \ref{lemma:matrix_concentration}.

\end{proof}

Next, we use the previous lemma to show that the expected energy of $pA-\mathbf{B}$ is large. This is somewhat counter-intuitive, as the previous lemma implies that none of the eigenvalues of $pA-\mathbf{B}$ is large. However, we show that the Frobenius norm of $pA-\mathbf{B}$ must be large in expectation, which tells us that the sum of the squares of the eigenvalues is large. Given none of the eigenvalues is large, this is only possible if the energy of $pA-\mathbf{B}$ is large as well.

\begin{lemma}\label{lemma:large_energy}
$$\mathbb{E}(\mathcal{E}(pA-\mathbf{B}))=\Omega\left(\frac{pm}{\sqrt{\Delta D}\log m}\right).$$
\end{lemma}

\begin{proof}
Let $\lambda_1\geq\dots\geq\lambda_n$ be the eigenvalues of $pA-\mathbf{B}$, let $\lambda=||pA-\mathbf{B}||$, and $t=20\log m \sqrt{\Delta D}$. Without loss of generality, we may assume that $G$ has no isolated vertices, so $m\leq 2n$. But then, by the previous lemma, $|\lambda_i|\leq \lambda\leq t$ with probability at least 
$$1-2n\exp\left(-\frac{t^2}{8p\Delta D+4Dt}\right).$$
Here, 
$$\frac{t^2}{8p\Delta D+4Dt}=\frac{400(\log m)^2 \Delta D}{8p\Delta D+80 D (\log m)\sqrt{\Delta D}}\geq \frac{400(\log m)^2 \Delta D}{8p\Delta D+80(\log m) \Delta D}> 4\log m.$$
Thus, we conclude
$$\mathbb{P}(\lambda>t)\leq  2nm^{-4}<m^{-2}.$$ 
Consider the Frobenius norm of $pA-\mathbf{B}$. We have
\begin{equation}\label{equ:frobenius}
    ||pA-\mathbf{B}||_F^2=\sum_{i=1}^{n}\lambda_i^2\leq \lambda\sum_{i=1}^{n}|\lambda_i|=\lambda\cdot \mathcal{E}(pA-\mathbf{B}).
\end{equation}
Let $\mathbf{Y}=||pA-\mathbf{B}||_F^2$.  If $u,v\in V(G)$, then $\mathbf{B}(u,v)=z_1\mathbf{I}_1+\dots+z_r\mathbf{I}_r$, where $\mathbf{I}_1,\dots,\mathbf{I}_r$ are the independent indicator random variables of colors appearing between $u$ and $v$ in $G$, and $z_1,\dots,z_r$ are the multiplicities of colors.  Note that $z_i\geq 1$ and $z_1+\dots+z_r=A(u,v)$. Hence, $$\mathbb{E}((pA(u,v)-\mathbf{B}(u,v))^2)=\mbox{Var}(\mathbf{B}(u,v))=\sum_{i=1}^{r}\mbox{Var}(z_i\mathbf{I}_r)=\sum_{i=1}^{r}(p-p^2)z_i^2\geq (p-p^2)A(u,v).$$
Therefore, 
$$\mathbb{E}(\mathbf{Y})\geq \sum_{u,v\in V(G)}(p-p^2) A(u,v)=2(p-p^2)m=\Omega(pm).$$
By (\ref{equ:frobenius}), we can write
$$\mathbb{E}(\mathcal{E}(pA-\mathbf{B}))\geq \mathbb{E}\left(\frac{\mathbf{Y}}{\lambda}\right).$$
In order to bound the right-hand-side, we condition on the event $(\lambda\leq t)$: 
$$\mathbb{E}\left(\frac{\mathbf{Y}}{\lambda}\right)\geq \mathbb{P}(\lambda\leq t)\cdot \mathbb{E}\left(\frac{\mathbf{Y}}{\lambda}| \lambda\leq t\right)\geq \frac{1}{t}\mathbb{P}(\lambda\leq t)\cdot \mathbb{E}(\mathbf{Y}|\lambda\leq t).$$
Here, using that $\mathbb{P}(\lambda>t)\leq m^{-2}$ and  $\max(\mathbf{Y})\leq m^2$, we can further write
$$\mathbb{P}(\lambda\leq t)\cdot \mathbb{E}(\mathbf{Y}|\lambda\leq t)=\mathbb{E}(\mathbf{Y})-\mathbb{P}(\lambda>t)\cdot\mathbb{E}(\mathbf{Y}|\lambda>t)\geq \mathbb{E}(\mathbf{Y})-\mathbb{P}(\lambda>t)\cdot\max(\mathbf{Y})\geq \Omega(pm).$$
Thus, $\mathbb{E}(\mathcal{E}(pA-\mathbf{B}))\geq \Omega(\frac{pm}{t})$, finishing the proof.
\end{proof}

Our final technical lemma shows that $M+\textbf{B}$ has large expected surplus for any fixed matrix $M$.

\begin{lemma}\label{lemma:ultimate}
Let $M\in \mathbb{R}^{V(G)\times V(G)}$ be symmetric with only zeros in the diagonal. Then
$$\mathbb{E}(\surp(M+\mathbf{B}))=\Omega\left(\frac{pm}{\sqrt{\Delta D}(\log m)^2}\right).$$
\end{lemma}

\begin{proof}
Let $$t=\frac{\varepsilon pm}{\sqrt{\Delta D}\log m},$$
where $\varepsilon>0$ is the constant hidden by the $\Omega(.)$ notation in Lemma \ref{lemma:large_energy}. Thus, $$\mathbb{E}(\mathcal{E}(\mathbf{B}-pA))\geq t.$$
First, assume that $\mathcal{E}(M+pA)>\frac{t}{8}$. Then by Corollary \ref{cor:surplus_energy}, we have 
$\surp(M+pA)=\Omega(\frac{t}{\log n})$.
This means that there exists a vector $x\in \{-1,1\}^{V(G)}$ such that 
$$\surp(M+pA)=-\frac{1}{2}\sum_{u,v\in V(G)}(M(u,v)+pA(u,v))x(u)x(v)=\Omega\left(\frac{t}{\log n}\right).$$
But then
\begin{align*}
    \mathbb{E}(\surp(M+\mathbf{B}))&\geq \mathbb{E}\left[-\frac{1}{2}\sum_{u,v\in V(G)}(M(u,v)+\mathbf{B}(u,v))x(u)x(v)\right]=\\
    &-\frac{1}{2}\sum_{u,v\in V(G)}(M(u,v)+pA(u,v))x(u)x(v)=\surp(M+pA).
\end{align*}
so we are done.

Hence, we may assume that $\mathcal{E}(M+pA)\leq \frac{t}{8}$. Writing $(\mathbf{B}-pA)=(M+\mathbf{B})+(-pA-M)$, we can apply Lemma \ref{lemma:energy_additive} to deduce that $\mathcal{E}(\mathbf{B}-pA)\leq 4(\mathcal{E}(M+\mathbf{B})+\mathcal{E}(M+pA))$. From this,
$$\mathbb{E}(\mathcal{E}(M+\mathbf{B}))\geq \mathbb{E}\left(\frac{1}{4}\mathcal{E}(\mathbf{B}-pA)-\mathcal{E}(M+pA)\right)\geq \frac{t}{8}.$$
 Therefore, by Corollary \ref{cor:surplus_energy},
$$\mathbb{E}(\surp(M+\mathbf{B}))\geq \Omega\left(\mathbb{E}\left(\frac{\mathcal{E}(M+\mathbf{B})}{\log n}\right)\right)=\Omega\left(\frac{pm}{\sqrt{\Delta D}(\log m)^2}\right).$$
\end{proof}

\section{MaxCut in 3-graphs}\label{sect:proofs}

In this section, we prove our main theorems. We start with the proof of Theorem \ref{thm:main}, which we restate here for the reader's convenience.

\begin{theorem}\label{lemma:main_cut}
 Let $H$ be a 3-multigraph. Assume that $H$ contains an induced subghypergraph with $m$ edges, maximum degree $\Delta$ and maximum co-degree $D$. Then 
    $$\surp_3(H)\geq \Omega\left(\frac{m}{\sqrt{\Delta D}(\log m)^2}\right).$$
\end{theorem}

\begin{proof}
Let $G$ be the underlying multi-graph of $H$. For every edge $\{u,v,w\}\in E(H)$, color the copy of $\{u,v\}\in E(G)$ coming from $\{u,v,w\}$ with color $w$. Let $S_0\subset V(H)$ be such that $H[S_0]$ has $m$ edges, maximum degree $\Delta$, and maximum co-degree $D$. Let $Q\subset S_0$ be an arbitrary set such that there are at least $m/3$ edges of $H[S_0]$ containing exactly one vertex of $Q$. There exists such a $Q$, as for a random set of vertices of $S_0$, sampled with probability $1/3$, there are $4m/9$ such edges in expectation. Let $S=S_0\setminus Q$.

Sample the vertices of $H$ independently with probability $p=1/3$, and let $\mathbf{X}$ be the set of sampled vertices. First, we condition on $\mathbf{X}\setminus Q$, that is, we reveal $\mathbf{X}$ outside of $Q$, and treat $\mathbf{X}\cap Q$ as the source of randomness. We use bold letters to denote those random variables that depend on $\mathbf{X}\cap Q$. Let us introduce some notation, see also Figure \ref{figure:1} for an illustration:
\begin{itemize}
    \item $V=S\setminus \mathbf{X}$,
    \item $G_0$ is the graph on vertex set $V$ which contains those edges of $G$, whose color is in $Q$,
    \item[] $A\in \mathbb{R}^{V\times V}$ is the adjacency matrix of $G_0$,
    \item $\mathbf{G}_0^*$ is the subgraph of $G_0$ in which we keep those edges, whose color is in $ \mathbf{X}\cap Q$,
    \item[] $\mathbf{B}\in \mathbb{R}^{V\times V}$ is the adjacency matrix of $G_0^*$,
    \item $G_1^*$ is the graph on vertex set $V$ which contains those edges of $G$, whose color is in $\mathbf{X}\setminus Q$,
    \item[] $M\in \mathbb{R}^{V\times V}$ is the adjacency matrix of $G_1^*$,
    \item $\mathbf{G}^*$ is the subgraph of $G[\mathbf{X}^c]$, whose edges have color in $\mathbf{X}$. (Here, $\mathbf{X}^c=V(H)\setminus \mathbf{X}$.)
\end{itemize}

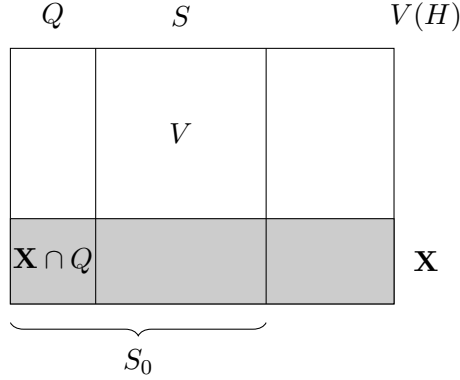
\begin{figure}
\begin{center}
  \begin{tikzpicture}[scale=0.85]
    \draw[] (-3,-2) rectangle (3,2);
    \draw[fill=black!20] (-3,-2) rectangle (3,-0.666) ;
    \draw[] (1,-2) -- (1,2);
    \draw[] (-1.666,-2) -- (-1.666,2);
    \node[] at (-2.333,-1.333) {$\mathbf{X}\cap Q$};
    \node[] at (-2.333,2.5) {$Q$};
    \node[] at (3.5,-1.333) {$\mathbf{X}$};
    \node[] at (-0.333,2.5) {$S$};
    \node[] at (-0.333,0.666) {$V$};
    \draw [decorate,decoration={brace,amplitude=5pt,mirror,raise=1.5ex}]
  (-3,-2) -- (1,-2) node[midway,yshift=-2em]{$S_0$};
      \node at (3.5,2.5) {$V(H)$};
  \end{tikzpicture}
\end{center}
\caption{An illustration of the different subsets of the vertex set used in the proof of Theorem \ref{lemma:main_cut}}
\label{figure:1}
\end{figure}

Note that $G_0$ has maximum degree at most $2\Delta$, as the degree of every vertex in $G_0$ is at most two times its degree in $H[S_0]$. Also, every color in $G_0$ appears at most $D$ times at every vertex. Indeed, if some color $q\in Q$ appears more than $D$ times at some vertex $v\in V$, then the co-degree of $qv$ in $H[S_0]$ is more than $D$, contradicting our choice of $S_0$. Thus, the matrices $A$ and $\mathbf{B}$ fit the Setup of Section \ref{sect:colored}.

Also, $\mathbf{G}^*[V]$ is the edge-disjoint union of $\mathbf{G}_0^*$ and $G_1^*$, so the adjacency matrix of $\mathbf{G}^*[V]$ is $M+\mathbf{B}$. Hence, $\surp(\mathbf{G}^*)\geq \surp(\mathbf{G}^*[V])=\surp(M+\mathbf{B})$, where the first inequality follows from Lemma \ref{lemma:surplus_sum}. Thus, writing $f=e(G_0)$, we can apply Lemma \ref{lemma:ultimate} to get
$$\mathbb{E}(\surp(\mathbf{G}^*)|\mathbf{X}\setminus Q)\geq \mathbb{E}(\surp(M+\mathbf{B})|\mathbf{X}\setminus Q)=
\Omega\left(\frac{f}{\sqrt{\Delta D}(\log f)^2}\right)$$
As we have $f\leq m$,
$$\mathbb{E}(\surp(\mathbf{G}^*)|\mathbf{X}\setminus Q)=\Omega\left(\frac{f}{\sqrt{\Delta D}(\log m)^2}\right).$$
Furthermore,
$$\mathbb{E}(\surp(\mathbf{G}^*))=\mathbb{E}(\mathbb{E}(\surp(\mathbf{G}^*)|\mathbf{X}\setminus Q))\geq \Omega\left(\frac{\mathbb{E}(f)}{\sqrt{\Delta D}(\log m)^2}\right).$$
Note that $f$ is the number of those edges in $G[S\setminus \mathbf{X}]$, whose color is in $Q$. For every edge of $G[S]$, whose color is in $Q$, the probability that it survives is $(1-p)^2p=\Omega(1)$. Hence, $\mathbb{E}(f)=\Omega(m)$. In conclusion, $$\mathbb{E}(\surp(\mathbf{G}^*))=\mathbb{E}(\mathbb{E}(\surp(\mathbf{G}^*)|\mathbf{X}\setminus Q))\geq \Omega\left(\frac{m}{\sqrt{\Delta D}(\log m)^2}\right).$$

Let $\mathbf{Y},\mathbf{Z}$ be a partition of $\mathbf{X}^c$ such that $e_{\mathbf{G}^*}(\mathbf{Y},\mathbf{Z})=\frac{1}{2}e(\mathbf{G}^*)+\surp(\mathbf{G}^*)$. Then $e_H(\mathbf{X},\mathbf{Y},\mathbf{Z})=e_{\mathbf{G}^*}(\mathbf{Y},\mathbf{Z})$, so
$$\mathbb{E}(e_H(\mathbf{X},\mathbf{Y},\mathbf{Z}))=\mathbb{E}(e_{\mathbf{G}^*}(\mathbf{Y},\mathbf{Z}))=\frac{1}{2}\mathbb{E}(e(\mathbf{G}^*))+\mathbb{E}(\surp(\mathbf{G}^*))=\frac{2}{9}e(H)+\Omega\left(\frac{m}{\sqrt{\Delta D}(\log m)^2}\right).$$
Therefore, choosing a partition $(\mathbf{X},\mathbf{Y},\mathbf{Z})$ achieving at least the expectation, we find that $$\surp_3(H)=\Omega\left(\frac{m}{\sqrt{\Delta D}(\log m)^2}\right).$$
\end{proof}

From this, Theorem \ref{thm:main_cut_linear} follows almost immediately. In particular, we prove the following slightly more general result.

\begin{theorem}\label{thm:cut_codegree}
Let $H$ be a 3-graph with $m$ edges and maximum co-degree $D$. Then
$$\surp_3(H)=\Omega\left(\frac{m^{3/4}}{D^{3/4}(\log m)^2}\right).$$
\end{theorem}

\begin{proof}
Let $G$ be the underlying graph of $H$, then $e(G)=3m$. Let $\Delta=100D^{1/2}m^{1/2}$, let $T$ be the set of vertices of $G$ of degree more than $\Delta$, and let $S=V(G)\setminus T$. As 
$6m=2e(G)\geq |T|\Delta$, we get that $|T|\leq 0.1m^{1/2}D^{-1/2}$. Hence, as the multiplicity of every edge of $G$ is at most $D$, we can write $e(G[T])\leq \frac{D}{2}|T|^2\leq \frac{1}{200}m$.

If $e_G(S,T)\geq 1.6m$, then $\surp(G)\geq 0.1 m$, so we are done as Lemma \ref{lemma:underlying} implies $\surp_3(H)\geq \frac{1}{6}\surp(G)$. 

Hence, we may assume that $e_G(S,T)<1.6m$. But then we must have $e(G[S])\geq 1.3m$, as $e(G[T])+e_G(S,T)+e(G[S])=e(G)=3m$. Moreover, 
$$e(G[S])\leq 3e(H[S])+e_H(S,T)\leq 3e(H[S])+\frac{1}{2}e_G(S,T),$$
from which we conclude that $e(H[S])\geq 0.1 m$. The maximum degree of $H[S]$ is less than the maximum degree of $G[S]$, which is less than $\Delta=100m^{1/2}D^{1/2}$. Also, the maximum co-degree of $H[S]$ is $D$, so applying Theorem \ref{lemma:main_cut} to the subhypergraph induced by $S$ finishes the proof.
\end{proof}

Next, we prove Theorem \ref{thm:main_cut}.

\begin{proof}[Proof of Theorem \ref{thm:main_cut}]
Let $G$ be the underlying graph of $H$, and consider $G$ as an edge weighted graph, where the weight of each edge is its multiplicity. We may assume that $\surp(G)\leq \frac{m^{3/5}}{(\log m)^2}$, otherwise we are done by Lemma \ref{lemma:underlying}. Let $D=m^{1/5}$. Say that an edge $e$ of $G$ is \emph{$D$-heavy} if its weight is at least $D$.

Let $S$ be the vertex set of a maximal matching of $D$-heavy edges.

\begin{claim}
    $|S|< m^{2/5}$
\end{claim}

\begin{proof}
Assume that $|S|\geq m^{2/5}$, and let $e_1,\dots,e_s$ be a matching of $D$-heavy edges, where $s=m^{2/5}/2$. By Lemma \ref{lemma:surplus_sum}, we have
$$\surp(G)\geq \sum_{i=1}^{s}\surp(G[e_i])\geq s\cdot \frac{D}{2}=\frac{m^{3/5}}{4},$$
contradiction.
\end{proof}

Let $\Delta=m^{3/5}$, and let $T$ be the set of vertices of $G$ of degree more than $\Delta$. Then $3m=e(G)\geq \frac{1}{2}|T|\Delta$, so $|T|\leq 6m^{2/5}$. Let $U=S\cup T$, then $|U|\leq 7m^{2/5}$.

\begin{claim}
   $e(G[U])\leq 0.1m$ 
\end{claim}
\begin{proof}
If $e(G[U])>0.1m$, then by Lemmas \ref{lemma:surplus_sum} and \ref{lemma:m/n}, we get
$$\surp(G)\geq \surp(G[U])= \Omega\left(\frac{e(G[U])}{|U|}\right)=\Omega(m^{3/5}),$$
contradiction.
\end{proof}

Let $W=U^c$. Then $G[W]$ contains no $D$-heavy edges, and the maximum degree of $G[W]$ is at most $\Delta$. Also, $e_G(U,W)\leq 1.6m$, as otherwise $\surp(G)\geq 0.1m$. This implies that $e_H(U,W)\leq 0.8m$ and that $e(G[W])=e(G)-e(G[U])-e_G(U,W)\geq 1.3m$. But then 
$$e(H[W])\geq  \frac{1}{3}(e(G[W])-e_H(U,W))>0.1m.$$ 
Thus, $H[W]$ is an induced subhypergraph of $H$ with at least $0.1m$ edges, maximum degree at most $\Delta$, and maximum co-degree at most $D$. Applying Lemma \ref{lemma:main_cut}, we conclude that
$$\surp_3(H)\geq  \Omega\left(\frac{m}{\sqrt{\Delta D}(\log m)^2}\right)=\Omega\left(\frac{m^{3/5}}{(\log m)^2}\right).$$
\end{proof}

Finally, we prove Theorem \ref{thm:rk}.

\begin{proof}[Proof of Theorem \ref{thm:rk}]
Let $H_r=H$, and for $i=r-1,\dots,3$, let $H_i$ be the underlying $i$-graph of $H_{i+1}$. Then by Lemma \ref{lemma:underlying}, $\surp_{i+1}(H_{i+1})\geq \frac{1}{i+1}\surp_{i}(H_i)$. From this, $\surp_r(H)\geq \Omega_r(\surp_3(H_{3}))$. But $H_3$ has $r(r-1)\dots 4\cdot m$ edges, so $\surp_3(H_3)=\Omega_r\left(\frac{m^{3/5}}{(\log m)^2}\right)$ by Theorem \ref{thm:main_cut}. Thus, $\surp_r(H)=\Omega_r\left(\frac{m^{3/5}}{(\log m)^2}\right)$ as  well.

In case $H$ is linear, the maximum co-degree of $H_3$ is $O_r(1)$, so applying Theorem \ref{thm:cut_codegree} gives the desired result.
\end{proof}

\section{Concluding remarks}

In this paper, we proved that every $r$-graph with $m$ edges contains an $r$-partite subhypergraph with at least $\frac{r!}{r^r}m+m^{3/5-o(1)}$ edges. This is still somewhat smaller than the conjectured  bound  $\frac{r!}{r^r}m+\Omega_r(m^{2/3})$ of Conlon, Fox, Kwan, and Sudakov \cite{CFKS}. In order to prove this conjecture, it is enough to prove the following strengthening of Theorem \ref{thm:main}.

\begin{conjecture}
If $H$ is a $3$-multigraph that contains an induced subhypergraph with $m$ edges and maximum degree $\Delta$, then $\surp_3(H)=\Omega(\frac{m}{\sqrt{\Delta}})$.
\end{conjecture}

 Indeed,  our Theorem \ref{thm:main} is sharp (up to logarithmic terms) as long as the maximum co-degree is $O(1)$. However, we believe the dependence on the co-degree in our lower bound is not necessary. On another note, we are unable to extend our methods to attack the problem of $k$-cuts in $r$-graphs in case $k\leq r-2$. The following remains and intriguing open problem.

\begin{conjecture}[\cite{CFKS}]
Let $2\leq k\leq r$ be integers such that $r\geq 3$ and $(r,k)\neq (3,2)$. If $H$ is an $r$-graph with $m$ edges, then $\surp_k(H)=\Omega_r(m^{2/3})$.
\end{conjecture}

In the case of linear hypergraphs, we proved that $\surp_r(H)>m^{3/4-o(1)}$, which is optimal up to the $o(1)$ term. It would be interesting to extend this result for $k$-cuts as well in case $k\leq r-2$, or to improve the $o(1)$ term.

\begin{conjecture}
Let $2\leq k\leq r$ be integers such that $r\geq 3$ and $(r,k)\neq (3,2)$. If $H$ is a linear $r$-graph with $m$ edges, then $\surp_k(H)=\Omega_r(m^{3/4})$.
\end{conjecture}

\section*{Acknowledgments}
We would like to thank Klas Markstr\"om for pointing out useful matrix concentration inequalities. Also, we would like to thank  Shengjie Xie and Jie Ma for pointing out a mistake in an earlier version of our paper.

\end{document}